\documentclass{amsart}
\usepackage{ amsmath, amsthm, amsfonts, hyperref, graphicx, color, ifpdf}

\newcommand{\ol}{\overline}

\newcommand{\be}{\begin{equation}}
\newcommand{\ee}{\end{equation}}

\newcommand{\lll}{\left}
\newcommand{\rrr}{\right}

\newcommand{\tgij}{\tilde{g}_{ij}}
\newcommand{\thij}{\tilde{h}_{ij}}
\newcommand{\tguij}{\tilde{g}^{ij}}

\newcommand{\tgukl}{\tilde{g}^{kl}}
\newcommand{\thh}{\tilde{h}}
\newcommand{\tg}{\tilde{g}}
\newcommand{\F}{(F-\sigma)}
\newcommand{\tna}{\tilde{\nabla}}
\newcommand{\na}{\nabla}
\newcommand{\ka}{\kappa}

\newtheorem{theorem}{Theorem}[section]
\newtheorem{lemma}[theorem]{Lemma}
\newtheorem{proposition}[theorem]{Proposition}
\newtheorem{corollary}[theorem]{Corollary}

\theoremstyle{definition}

\theoremstyle{remark}
\newtheorem{remark}[theorem]{Remark}

\numberwithin{equation}{section}

\begin{document}
\setlength{\baselineskip}{1.2\baselineskip}

\title[Modified general curvature flow in hyperbolic space II]
{curvature flow of complete hypersurfaces in hyperbolic space}
%with asymptotic boundary at infinity}

\author{Ling Xiao}
%\address{Department of Mathematics\\Johns Hopkins University\\3400 North Charles Street\\Baltimore, MD 21218-2686\\USA}
%\email{lxiao@math.jhu.edu}

%\thanks{}

%    \subjclass is required.
%\subjclass[2000]{}
%    The 2010 edition of the Mathematics Subject Classification is
%    now available.  If you are citing a classification from the
%    new scheme, use the following input coding instead.

%\date{\today}

\begin{abstract}
In this paper we continue our study of finding the curvature flow of complete hypersurfaces in hyperbolic space with a prescribed asymptotic boundary at infinity. Our main results are proved by deriving a priori global gradient estimates and $C^2$ estimates.
\end{abstract}

\maketitle

\section{Introduction} \label{Int}
In this paper, we continue our study of the modified curvature flow problem in hyperbolic space $\mathbb{H}^{n+1}$. Consider a complete Weingarten hypersurface in $\mathbb{H}^{n+1}$ with a prescribed asymptotic boundary $\Gamma$ at infinity, whose principal curvatures satisfy $f(\kappa[\Sigma_0])\leq\sigma$ (e.g. we can choose a constant mean curvature graph as found in \cite{NS96}),
and is given by an embedding $\mathbf{X}(0): \Omega\rightarrow\mathbb{H}^{n+1}$, where $\Omega\subset\partial_{\infty}\mathbb{H}^{n+1}.$ We consider the evolution of such an embedding to produce a family of embeddings $\mathbf{X}:\Omega\times[0,T)\rightarrow\mathbb{H}^{n+1}$ satisfying the following equations
\begin{equation}{\label{Int0}}
\left\{
\begin{aligned}
&\dot{\mathbf{X}}=(f(\ka[\Sigma])-\sigma)\nu_H\,\,&\mbox{in $\Omega\times[0,T)$},\\
&\mathbf{X}=\Gamma\,\,&\mbox{on $\partial\Omega\times[0,T)$},\\
&\mathbf{X}(0)=\Sigma_0\,\,&\mbox{in $\Omega\times\{0\}$},
\end{aligned}
\right.
\end{equation}
where $\ka[\Sigma(t)]=(\ka_1,\cdots,\ka_n)$ denotes the hyperbolic principal curvatures of $\Sigma(t),$ $\sigma\in (0,1)$ is a constant, and $\nu_{H}$ denotes the outward unit normal of $\Sigma(t)$ with respect to the hyperbolic metric.

In this paper we shall use the half-space model
\[\mathbb{H}^{n+1}=\{(x,x_{n+1})\in\mathbb{R}^{n+1}:x_{n+1}>0\}\]
equipped with the hyperbolic metric
\be\label{Int4}
ds^2=\frac{\sum_{i=1}^{n+1}dx_i^2}{x_{n+1}^2}.
\ee
One identifies the hyperplane $\{x_{n+1}=0\}=\mathbb{R}^n\times\{0\}\subset\mathbb{R}^{n+1}$ as the infinity of
$\mathbb{H}^{n+1},$ denoted by $\partial_{\infty}\mathbb{H}^{n+1}.$ For convenience we say $\Sigma$ has compact asymptotic boundary if $\partial\Sigma\subset\partial_{\infty}\mathbb{H}^{n+1}$ is compact with respect to the Euclidean metric in $\mathbb{R}^n.$

We assume the function $f$ satisfies the following fundamental structure conditions:
\be\label{Int5}
f_i(\lambda)\equiv\frac{\partial f(\lambda)}{\partial\lambda_i}>0\;\;\mbox{in $K$},\;\;1\leq i\leq n,
\ee
\be\label{Int6}
\mbox{$f$ is a concave function in $K$},
\ee
and
\be\label{Int7}
f>0\;\;\mbox{in $K$},\;\;f=0\;\;\mbox{on $\partial K$},
\ee
 where $K\subset\mathbb{R}^n$ is an open symmetric convex cone such that
 \be\label{Int8}
 K^+_n:=\{\lambda\in\mathbb{R}^n:\mbox{each component $\lambda_i>0$}\}\subset K.
 \ee
 In addition, we shall assume that $f$ is normalized
 \be\label{Int9}
 f(1,\cdots,1)=1
 \ee
 and satisfies the more technical assumptions
 \be\label{Int10}
 \mbox{$f$ is homogeneous of degree one}
 \ee
 and
 \be\label{Int11}
 \lim_{R\rightarrow +\infty}f(\lambda_1,\cdots,\lambda_{n-1},\lambda_{n}+R)\geq 1+\epsilon_0\;\;\;\mbox{uniformly in $B_{\delta_0}(\mathbf{1})$}
 \ee
 for some fixed $\epsilon_0>0$ and $\delta_0>0,$ where $B_{\delta_0}(\mathbf{1})$ is the ball centered at $\mathbf{1}=(1,\cdots,1)\in\mathbb{R}^n.$

 As we can see in \cite{GS08}, an example of a function satisfying all of these assumptions above is given by
 $f=(H_k/H_l)^{\frac{1}{k-l}},\;\;\mbox{$0\leq l<k$},$ defined in $K,$ where $H_l$ is the normalized $l-th$ elementary symmetric polynomial.(e.g, $H_0=1,$ $H_1=H,$ $H_n=K$ the extrinsic Gauss curvature.)

 Since $f$ is symmetric, from (\ref{Int6}), (\ref{Int9}) and (\ref{Int10}) we have
 \be\label{Int12}
 f(\lambda)\leq f(\mathbf{1})+\sum f_i(\mathbf{1})(\lambda_i-1)=\sum f_i(\mathbf{1})\lambda_i=\frac{1}{n}\sum\lambda_i\;\;\mbox{in $K$}
 \ee
and
\be\label{Int13}
\sum f_i(\lambda)=f(\lambda)+\sum f_i(\lambda)(1-\lambda_i)\geq f(\mathbf{1})=1\;\;\mbox{in $K$.}
\ee

In this paper, we always assume the initial surface $\Sigma_0$ to be connected and orientable, $\Sigma(t)=\{(x,u(x,t))|x\in\Omega,\;x_{n+1}=u,\;0\leq t<T\}$ to be the flowing surfaces, and the position vector $\mathbf{X}=(x,u(x,t))$ satisfies the flow equation (\ref{Int0}). If $\Sigma$ is a complete hypersurface in $\mathbb{H}^{n+1}$ with compact asymptotic boundary at infinity, then the normal vector field of $\Sigma$ is always chosen to be the one pointing to the unique unbounded region in $\mathbb{R}^{n+1}_+/\Sigma.$ In this case, both the Euclidean and hyperbolic principal curvature of $\Sigma$ are calculated with respect to this normal field.

 We shall take $\Gamma=\partial\Omega,$ where $\Omega\subset\mathbb{R}^n$ is a smooth domain and let $\Gamma_\epsilon$ denote its vertical lift. We seek a family of hypersurfaces $\Sigma(t)$ as the graph of a function $u(x,t)$ with boundary $\Gamma$ satisfying equation (\ref{Int0}). Then the coordinate vector fields and upper unit normal are given by
 \[\mathbf{X}_i=e_i+u_ie_{n+1},\;\;\nu_H=u\nu=u\frac{-u_ie_i+e_{n+1}}{w},\]
 where through out this paper $w=\sqrt{1+|\na u|^2}$ and $e_{n+1}$ is the unit vector in the positive $x_{n+1}$ direction in $\mathbb{R}^{n+1}.$

Notice that
 \[\left<\dot{\mathbf{X}}, \nu_H\right>_H=f-\sigma,\]
which is equivalent to
\[\left<\frac{\partial}{\partial t}(x, u(x, t)), \nu_H\right>_H=f-\sigma.\]
Thus the height function $u$ satisfies equation
\be\label{Int1} u_t=(F-\sigma)uw.\ee

Therefore problem (\ref{Int0})  can be represented as the Dirichlet problem for a fully nonlinear second order parabolic equation
\be\label{Int2}
\left\{
\begin{aligned}&u_t=uw(f-\sigma) &\mbox{in $\Omega\times [0, T)$\,,}\\
&u(x, t)=0 &\mbox{on $\partial\Omega\times [0, T)$\,,}\\
&u(x, 0)=u_0 &\mbox{in $\Omega\times \{0\}$\,.}
\end{aligned}
\right.
\ee

In this paper, we shall focus on proving the long time existence of the modified general curvature flow (MGCF) of a complete embeded hypersurface
with initial surface $\Sigma_0=\{(x, u_0(x)),\;x\in\Omega\}$ satisfying $f(\kappa[\Sigma_0])\leq\sigma$ and $1/w(u_0)>\sigma.$ These additional assumptions
will be needed in the proof of Proposition \ref{Grep0}. (Note that for constant mean curvature graph the latter assumption is trivial.)
Following the literature we define the class of $admissible$ $functions:$
$$\mathcal{A}(\Omega_T)=\{u\in C^{2,1}(\Omega_T):\kappa[u]\in K\}.$$
 Condition (\ref{Int5}) implies that equation (\ref{Int2}) is parabolic for admissible solutions. Our goal is to show that the Dirichlet problem (\ref{Int2}) admits smooth admissible solutions for all $0<\sigma<1.$ Due to the special nature of the problem we saw in \cite{GS08}, there are substantial technical difficulties to overcome and we have not yet succeeded in finding the solutions for all $\sigma\in (0,1).$ However, we succeed in improving the result in \cite{GS08}.

\begin{theorem}\label{Intt0}
Let $\Gamma=\partial\Omega\times\{0\}\subset\mathbb{R}^{n+1}$ where $\Omega$ is a bounded smooth domain in $\mathbb{R}^n.$ Suppose that the Euclidean mean curvature $\mathcal{H}_{\partial\Omega}\geq 0$ and $\sigma\in(0,1)$ satisfies $\sigma>\sigma_0,$ where $\sigma_0$ is the unique zero in $(0,1)$ of
\be\label{Int14}
\phi(a):=\frac{4}{3}a-\frac{1}{27}a^3-\frac{1}{27}(a^2+3)^{3/2}.
\ee
(Numerical calculations show $0.14596<\sigma_0<0.14597.$)

Under conditions (\ref{Int5})--(\ref{Int11}), there exists a solution $\Sigma(t)$, $t\in[0,\infty)$, to the MGCF (\ref{Int0}) with uniformly bounded principal curvatures
\be\label{Int15}
|\kappa[\Sigma(t)]|\leq C\;\; \mbox{on $\Sigma(t)$, $\forall t\in[0,\infty)$.}
\ee
Moreover, $\Sigma(t)=\{(x,u(x,t))\mid (x,t)\in\Omega\times[0,\infty)\}$ is the flowing surfaces of the unique admissible solution $u(x,t)\in C^\infty(\Omega\times(0,\infty))\cap W^{2,1}_p(\Omega\times[0,\infty))$ of the Dirichlet problem (\ref{Int2}), where $p>4.$ Furthermore, for any fixed $t\in[0,\infty),$ $u^2(x,t)\in C^{\infty}(\Omega)\cap C^{1+1}(\ol\Omega)$ and
\be\label{Int16}
\sqrt{1+|Du|^2}\leq C\;\;\mbox{in $\Omega$},
\ee
\be\label{Int17}
u|D^2u|\leq C\;\;\mbox{in $\Omega$}.
\ee
In addition, as $t\rightarrow\infty,$ $u(x,t)$ converges uniformly to a function $\tilde{u}(x)\in C^{\infty}(\Omega)\cap C^1(\ol\Omega)$ such that
$\Sigma_{\infty}=\{(x,\tilde{u}(x))\mid x\in\Omega\}$ is a unique complete surface satisfies $f(\kappa[\Sigma_\infty])=\sigma$ in $\mathbb{H}^{n+1}.$
\end{theorem}

 Equation (\ref{Int2}) is degenerate when $u=0.$ It is therefore very natural to approximate the boundary condition
$u=0\;\mbox{on $\partial{\Omega}\times[0,T)$}$ by $u=\epsilon\;\;\mbox{on $\partial{\Omega}\times[0,T)$}$, for $\epsilon>0$ sufficiently small. So the problem becomes
\be\label{Int3}
\left\{
\begin{aligned}
&u_t=uw(f-\sigma)&\mbox{in $\Omega\times [0, T),$}\\
&u(x, t)=\epsilon&\mbox{on\, $\partial\Omega\times [0, T),$}\\
&u(x, 0)=u_0^\epsilon=u_0+\epsilon&\mbox{in $\Omega\times \{0\},$}
\end{aligned}
\right.
\ee
where $\Sigma_0^\epsilon=\{(x,u_0^\epsilon)|x\in\Omega\}$ satisfies $f(\ka[\Sigma_0^\epsilon])\leq\sigma$ and $\frac{1}{w(u_0^\epsilon)}>\sigma,$
$\forall x\in\Omega.$

\begin{theorem}\label{Intt2}
Let $\Omega$ be a bounded smooth domain in $\mathbb{R}^n$ with $\mathcal{H}_{\partial\Omega}\geq 0$ and suppose $f$ satisfies (\ref{Int5})--(\ref{Int11}). Then for any $\sigma\in(0,1)$ and $\epsilon>0$ sufficiently small, there exists a unique admissible solution $u^\epsilon\in C^\infty(\ol\Omega\times (0,\infty))$ of the Dirichlet Problem (\ref{Int3}). Moreover, $u^\epsilon$ satisfies the a priori estimates
\be\label{Int18}
\sqrt{1+|Du^\epsilon|^2}\leq C\;\;\mbox{in $\Omega\times[0,\infty)$},
\ee
\be\label{Int19}
u^\epsilon|D^2u^\epsilon|\leq C\;\;\mbox{on $\partial\Omega\times[0,\infty)$},
\ee
and
\be\label{Int20}
u^\epsilon|D^2u^\epsilon|\leq C(\epsilon, t)\;\;\mbox{in $\Omega\times[0,\infty)$}.
\ee
In particular, $C(\epsilon,t)$ depends exponentially on time $t.$
\end{theorem}
\begin{remark}
The a priori estimates (\ref{Int18}) and (\ref{Int19}) will be proved in section \ref{Gre} and \ref{C2b}, while (\ref{Int20}) can be derived
by combining Theorem \ref{Sel0} and Lemma \ref{C2glm0} with the standard maximum principle for parabolic equations.
\end{remark}
The paper is organized as follows. In Section \ref{Foh} we establish some basic identities for hypersurface in $\mathbb{H}^{n+1}.$ In Section \ref{Se} we state the short time existence theorem and derive evolution equations for some geometric quantities. In Section \ref{Gre} we use the mean convex condition on the boundary to establish a sharp global gradient bound for $u$. In Section \ref{C2b} we show the boundary second derivative estimates. In Section \ref{c2g} we prove a maximum principle for the maximal hyperbolic principal curvature by using radial graphs (this idea is from \cite{GSZ09}). Finally in Section \ref{Con} we prove that as $t\rightarrow\infty,$ $\Sigma(t)$ converges uniformly to a hypersurface $\tilde{\Sigma}$ satisfies $f(\kappa[\tilde{\Sigma}])=\sigma.$

\section{Formulas for hyperbolic principal curvatures}\label{Foh}
\subsection{Formulas on hypersurfaces}\label{Fohfirst}
 We will compare the induced hyperbolic and Euclidean metrics and derive some basic identities on a hypersurface.

Let $\Sigma$ be a hypersurface in $\mathbb{H}^{n+1}.$ We shall use $g$ and $\na$ to denote the induced hyperbolic metric and Levi-Civita connections on $\Sigma,$ respectively. Since $\Sigma$ also can be viewed as a submanifold of $\mathbb{R}^{n+1},$ we shall usually identify a geodesic quantity with respect to the Euclidean metric by adding a 'tilde' over the corresponding hyperbolic quantity. For instance, $\tg$ denotes the induced metric on $\Sigma$ from $\mathbb{R}^{n+1},$ and $\tna$ is its Levi-Civita connection.

Let $(z_1,\cdots,z_n)$ be local coordinates and
$$\tau_i=\frac{\partial}{\partial z_i},\;\;\mbox{$i=1,\cdots,n$}.$$
The hyperbolic and Euclidean metrics of $\Sigma$ are given by
\be\label{Foh0}
g_{ij}=\lll<\tau_i, \tau_j\rrr>_H,\;\;\tg_{ij}=\tau_i\cdot\tau_j=u^2g_{ij},
\ee
while the second fundamental forms are
\be\label{Foh1}
\begin{aligned}
&h_{ij}=\lll<D_{\tau_i}\tau_j, \nu_H\rrr>_H=-\lll<D_{\tau_i}\nu_H,\tau_j\rrr>_H,\\
&\thh_{ij}=\nu\cdot\tilde{D}_{\tau_i}\tau_j=-\tau_j\cdot\tilde{D}_{\tau_i}\nu,\\
\end{aligned}
\ee
where $D$ and $\tilde{D}$ denote the Levi-Civita connection of $\mathbb{H}^{n+1}$ and $\mathbb{R}^{n+1},$ respectively.
The following relations are well known (see equation(1.7),(1.8) of \cite{GS08} ):
\be\label{Foh2}
h_{ij}=\frac{1}{u}\thh_{ij}+\frac{\nu^{n+1}}{u^2}\tg_{ij}.
\ee
\be\label{Foh3}
\ka_i=u\tilde{\ka}_i+\nu^{n+1},\;\;\mbox{$i=1,\cdots,n,$}
\ee
where $\nu^{n+1}=\nu\cdot e_{n+1}.$

The Christoffel symbols are related by formula
\be\label{Foh9}
\Gamma^k_{ij}=\tilde{\Gamma}^k_{ij}-\frac{1}{u}(u_i\delta_{kj}+u_j\delta_{ik}-\tg^{kl}u_l\tg_{ij}).
\ee
It follows that for $v\in C^2(\Sigma)$
\be\label{Foh10}
\na_{ij}v=v_{ij}-\Gamma^k_{ij}v_k=\tna_{ij}v+\frac{1}{u}(u_iv_j+u_jv_i-\tg^{kl}u_kv_l\tg_{ij})
\ee
where and in the sequel (if no additional explanation)
\[v_i=\frac{\partial v}{\partial x_i},\;v_{ij}=\frac{\partial^2v}{\partial x_i\partial x_j},\;etc.\]
In particular,
\be\label{Foh11}
\na_{ij}u=\tna_{ij}u+\frac{2u_iu_j}{u}-\frac{1}{u}\tg^{kl}u_ku_l\tg_{ij}.
\ee
Moreover in $\mathbb{R}^{n+1},$
\be\label{Foh12}
\tg^{kl}u_ku_l=|\tna u|^2=1-(\nu^{n+1})^2
\ee
\be\label{Foh13}
\tna_{ij}u=\thh_{ij}\nu^{n+1}.
\ee
We note that all formulas listed above still hold for general local frame $\tau_1, \cdots, \tau_n.$ In particular, if $\tau_1, \cdots, \tau_n$ are orthonormal  in the hyperbolic metric, then $g_{ij}=\delta_{ij}$ and $\tg_{ij}=u^2\delta_{ij}.$

We now consider equation (\ref{Int0}) on $\Sigma.$ For $K$ as in section 1, let $\mathcal{A}$ be the vector space of $n\times n$ matrices and
\[\mathcal{A}_K=\left\{A=\{a_{ij}\}\in\mathcal{A}: \lambda(A)\in K\right\},\]
where $\lambda(A)=(\lambda_1,\cdots,\lambda_n)$ denotes the eigenvalues of $A.$ Let $F$ be the function defined by
\be\label{Foh4}
F(A)=f(\lambda(A)),\;\;A\in\mathcal{A}_K
\ee
and denote
\be\label{Foh5}
F^{ij}(A)=\frac{\partial F}{\partial a_{ij}}(A),\;\;F^{ij,kl}(A)=\frac{\partial^2F}{\partial a_{ij}\partial a_{kl}}(A).
\ee
Since $F(A)$ depends only on the eigenvalues of $A,$ if $A$ is symmetric then so is the matrix $\left\{F^{ij}(A)\right\}.$ Moreover,
\[F^{ij}(A)=f_i\delta_{ij}\]
when $A$ is diagonal, and
\be\label{Foh6}
F^{ij}(A)a_{ij}=\sum f_i(\lambda(A))\lambda_i=F(A),
\ee
\be\label{Foh7}
F^{ij}(A)a_{ik}a_{jk}=\sum f_i(\lambda(A))\lambda^2_i.
\ee
Equation (\ref{Int2}) can therefore be rewritten in a local frame $\tau_1,\cdots, \tau_n$ in the form
\be\label{Foh8}
\left\{
\begin{aligned}
&u_t=uw(F(A[\Sigma])-\sigma)&\,\,(x,t)\in\Omega\times[0,T),\\
&u(x,t)=0&\,\,(x,t)\in\partial\Omega\times[0,T),\\
&u(x,0)=u_0&\,\,(x,t)\in\Omega\times\{0\},
\end{aligned}
\right.
\ee
where $A[\Sigma]=\lll\{g^{ik}h_{kj}\rrr\}.$ Let $F^{ij}=F^{ij}\lll(A[\Sigma]\rrr),$ $F^{ij,kl}=F^{ij,kl}\lll(A[\Sigma]\rrr).$
\subsection{Vertical graphs}\label{Fohsecond} Suppose $\Sigma$ is locally represented as the graph of a function $u\in C^2(\Omega),$ $u>0,$ in a domain $\Omega\subset\mathbb{R}^n:$
\[\Sigma=\{(x,u(x))\in\mathbb{R}^{n+1}: x\in\Omega\}.\]
In this case we take $\nu$ to be the upward (Euclidean) unit normal vector field to $\Sigma:$
\[\nu=\lll(-\frac{Du}{w},\frac{1}{w}\rrr),\;w=\sqrt{1+|Du|^2}.\]
The Euclidean metric and second fundamental form of $\Sigma$ are given respectively by
\[\tg_{ij}=\delta_{ij}+u_iu_j,\]
and
\[\thh_{ij}=\frac{u_{ij}}{w}.\]
 As shown in \cite{CNS86}, the Euclidean principal curvature $\tilde{\kappa}[\Sigma]$ are the eigenvalues of symmetric matrix $\tilde{A}[u]=[\tilde{a}_{ij}]:$
\be\label{Foh14}
\tilde{a}_{ij}:=\frac{1}{w}\gamma^{ik}u_{kl}\gamma^{lj},
\ee
where
\[\gamma^{ij}=\delta_{ij}-\frac{u_iu_j}{w(1+w)}.\]
Note that the matrix $\{\gamma^{ij}\}$ is invertible with the inverse
\[\gamma_{ij}=\delta_{ij}+\frac{u_iu_j}{1+w}\]
which is the square root of $\{\tilde{g}_{ij}\},$ i.e., $\gamma_{ik}\gamma_{kj}=\tilde{g}_{ij}.$ From (\ref{Foh3}) we see that the hyperbolic principal curvatures $\kappa[u]$ of $\Sigma$ are eigenvalues of the matrix $A[u]=\{a_{ij}[u]\}:$
\be\label{Foh15}
a_{ij}:=\frac{1}{w}\lll(\delta_{ij}+u\gamma^{ik}u_{kl}\gamma^{lj}\rrr).
\ee
When $\Sigma$ is a vertical graph we can also define $F(A[\Sigma])=F(A[u]).$

\subsection{Radial graphs}\label{Fohthird} Let $\nabla'$ denote the covariant derivative on the standard unit sphere $\mathbb{S}^n$ in $\mathbb{R}^{n+1}$ and $y=e_{n+1}\cdot\mathbf{z}$ for $\mathbf{z}\in\mathbb{S}^{n}\subset\mathbb{R}^{n+1}.$ Let $\tau_1, \cdots, \tau_n$ be a local frame of smooth vector fields on the upper hemisphere $\mathbb{S}_+^n$ and denote $\tau_i\cdot\tau_j=\sigma_{ij}.$

Suppose that locally $\Sigma$ is a radial graph over the upper hemisphere $\mathbb{S}^n_+\subset\mathbb{R}^{n+1},$ i.e., it is locally represented as
\be\label{Foh16}
\mathbf{X}=e^v\mathbf{z},\;\;\mathbf{z}\in\mathbb{S}^n_+\subset\mathbb{R}^{n+1}.
\ee
The Euclidean metric, outward unit normal vector and second fundamental form of $\Sigma$ are
\be\label{Foh17}
\tg_{ij}=e^{2v}(\sigma_{ij}+\nabla'_iv\nabla'_jv),
\ee
\be\label{Foh18}
\nu=\frac{\mathbf{z}-\nabla'v}{w},\;\;w=(1+|\nabla'v|^2)^{1/2},
\ee
and
\be\label{Foh19}
\thh_{ij}=\frac{e^v}{w}(\nabla'_{ij}v-\nabla'_iv\nabla'_jv-\sigma_{ij})
\ee
respectively. Therefore the Euclidean principal curvatures are the eigenvalues of the matrix
\be\label{Foh20}
\tilde{a}_{ij}=\frac{e^{-v}}{w}\lll(\gamma^{ik}\nabla'_{kl}v\gamma^{lj}-\sigma_{ij}\rrr),
\ee
where
\[\gamma^{ij}=\sigma^{ij}-\frac{v^iv^j}{w(1+w)}\] and $v^i=\sigma^{ik}\nabla'_kv.$ Note that the height function is $u=ye^v.$ We see that the hyperbolic principal curvatures are the eigenvalues of matrix $A^s[v]=\{a^s_{ij}[v]\}:$
\be\label{Foh21}
a^s_{ij}[v]:=\frac{1}{w}\lll(y\gamma^{ik}\nabla'_{kl}v\gamma^{lj}-\mathbf{e}\cdot\nabla'v\sigma_{ij}\rrr).
\ee
In this case we can define $F(A[\Sigma])=F(A^s[v]).$

\section{Short time existence and Evolution equations} \label{Se}
\subsection{Short time existence}
\label{Se-sho}
In order to prove a global existence for the Dirichlet problem (\ref{Int3}), we first need a short time existence theorem.
Here we shall apply Theorem 3.1 of \cite{LX11} directly. For completeness let's restate the theorem as following:

\begin{theorem}\label{Sel0}
Let $G(D^2u, Du, u)$ be a nonlinear operator that is smooth with respect to $D^2u, Du$ and $u.$
Suppose that $G$ is defined for a function $u$ belonging to an open set $\Lambda\subset C^2(\Omega)$
and $G$ is elliptic for any $u\in\Lambda,$ i.e., $G^{ij}>0,$ then the initial value problem
\be\label{Se0}
\left\{
\begin{aligned}
&u_t=G(D^2u,Du,u)\;\;&\mbox{in $\Omega\times [0, T^*)$},\\
&u(x, 0)=u_0 \;\;&\mbox{in $\Omega\times \{0\}$},\\
&u(x, t)=0\;\; &\mbox{on $\partial\Omega\times [0, T^*)$},
\end{aligned}
\right.
\ee
has a unique smooth solution $u$ when $T^*=\epsilon >0$ small enough,
except for the corner, where $u_0\in\Lambda$ be of class $C^{\infty}(\ol\Omega).$
\end{theorem}

Since Theorem \ref{Sel0} was proved in \cite{LX11}, we omit the proof here.

\medskip

\subsection{Evolution equations for some geometric quantities}\label{Evo}
For the reader's convenience, we now compute the evolution equations for some affine geometric quantities that were first derived in \cite{LX11}.
In this section we shall write $F_{ij}=\tna_{ij}F, u_{ij}=\tna_{ij}u, F^i_j=\tg^{il}F_{lj},$ etc.

\begin{lemma}\label{Evol0}(Evolution of the metrics). The metric $g_{ij}$ and $\tg_{ij}$ of $\Sigma(t)$ satisfies the evolution equations
\end{lemma}
\be\label{Evo8}\dot{g}_{ij}=-2u^{-2}\tgij(F-\sigma)w-2u^{-1}(F-\sigma)\thij,\ee
and
\be\label{Evo7}\dot{\tilde{g}}_{ij}=-2(F-\sigma)u\tilde{h}_{ij}.\ee
\begin{proof}
Since $\tgij=\tau_i\cdot\tau_j,$
\begin{align*}
&\frac{\partial}{\partial t}\tgij=2\lll<\tilde{D}_{\tau_i}\dot{X},\tilde{D}_{\tau_j}X\rrr>\\
&=2\lll<\tilde{D}_{\tau_i}[(F-\sigma)u\nu],\tau_j\rrr>\\
&=2(F-\sigma)u\lll<\tilde{D}_{\tau_i}\nu,\tau_j\rrr>\\
&=-2(F-\sigma)u\thij.
\end{align*}
From equation (\ref{Foh0}) we get
\begin{align*}
\frac{\partial}{\partial t}g_{ij}&=-2u^{-3}\tgij u_t+u^{-2}\dot{\tilde{g}}_{ij}\\
&=-2u^{-3}\tgij(F-\sigma)uw-2u^{-2}(F-\sigma)u\thij\\
&=-2u^{-2}\tgij(F-\sigma)w-2u^{-1}(F-\sigma)\thij.
\end{align*}
\end{proof}

\begin{lemma}\label{Evol1}(Evolution of the normal). The normal vector evolves according to
\be\label{Evo9} \dot{\nu}=-\tg^{ij}[(F-\sigma)u]_i\tau_j,\ee
moreover,
\be\label{Evo10}\dot{\nu}^{n+1}=-\tilde{g}^{ij}[(F-\sigma)u]_iu_j.\ee
\end{lemma}
\begin{proof}
Since $\nu$ is the unit normal vector of $\Sigma$, we have $\dot{\nu}\in T(\Sigma).$ Furthermore, differentiating
$$\lll<\nu,\tau_i\rrr>=\lll<\nu,\tilde{D}_{\tau_i}X\rrr>=0,$$
with respect to $t$ we deduce
\begin{align*}
\lll<\dot{\nu},\tau_i\rrr>&=-\lll<\nu,\tilde{D}_{\tau_i}[(F-\sigma)u\nu]\rrr>\\
&=-\lll<\nu,[(F-\sigma)u]_i\nu\rrr>\\
&=-[(F-\sigma)u]_i.
\end{align*}
So we have
\[\dot{\nu}=-\tguij[(F-\sigma)u]_i\tau_j,\]
and (\ref{Evo10}) follows directly from
\[\dot{\nu}^{n+1}=\lll<\dot{\nu},\mathbf{e}\rrr>=-\tguij[(F-\sigma)u]_iu_j.\]
\end{proof}

\begin{lemma}\label{Evol2}(Evolution of the second fundamental form). The second fundamental form evolves according to
\be\label{Evo11}\dot{\tilde{h}}^l_i=[(F-\sigma)u]^l_i+u(F-\sigma)\tilde{h}^k_i\tilde{h}^l_k,\ee
\be\label{Evo12}\dot{\tilde{h}}_{ij}=[(F-\sigma)u]_{ij}-u(F-\sigma)\tilde{h}^k_i\tilde{h}_{kj},\ee
and
\be\label{Evo13}\begin{aligned}&\dot{h}_{ij}=\frac{1}{u}\{[(F-\sigma)u]_{ij}-u(F-\sigma)\tilde{h}^k_i\tilde{h}_{kj}\}
-\frac{\thij}{u}w(F-\sigma)\\
&-\{\tg^{kl}[u(F-\sigma)]_ku_l\}\frac{\tgij}{u^2}
-2\frac{(F-\sigma)\nu^{n+1}}{u}\thij-2\frac{\tgij}{u^2}(F-\sigma).\end{aligned}\ee
\end{lemma}
\begin{proof}
Differentiating (\ref{Evo9}) with respect to $\tau_i$ we get
\[\frac{\partial}{\partial t}\nu_i=-\tgukl[(F-\sigma)u]_{ki}\tau_l
-\tgukl[(F-\sigma)u]_k\tilde{D}_{\tau_i}\tau_l.\]
On the other hand, in view of the Weingarten Equation
\[\nu_i=-\tg^{kl}\thh_{li}\tau_k\Rightarrow\dot{\nu_i}=-\dot{\thh}^k_i\tau_k-\thh^k_i\tilde{D}_{\tau_k}\dot{X},\]
where $\thh^k_i=\tg^{kl}\thh_{li}$ is mixed tensor.
Multiplying by $\tau_j$ we get
\[-\dot{\thh}^k_i\tg_{kj}-\thh^k_i\lll<\tilde{D}_{\tau_k}\dot{X},
\tau_j\rrr>=-\tgukl[(F-\sigma)u]_{ki}\tg_{lj}.\]
Thus
\begin{align*}\dot{\thh}^k_i\tg_{kj}&=\tgukl[(F-\sigma)u]_{ki}\tg_{lj}-\thh^k_iu(F-\sigma)\lll<\tilde{D}_{\tau_k}\nu,
\tau_j\rrr>\\
&=[(F-\sigma)u]_{ij}+u(F-\sigma)\thh^k_i\thh_{kj}.\end{align*}
Multiplying by $\tg^{jl}$ we get (\ref{Evo11}).

Moreover, since $\thij=\thh^l_i\tg_{lj},$ by equation (\ref{Evo7}) and (\ref{Evo11}) we have
\begin{align*}
\dot{\thh}_{ij}&=\dot{\thh}^l_i\tg_{lj}+\thh^l_i\dot{\tg}_{lj}\\
&=[(F-\sigma)u]^l_i\tg_{lj}+u(F-\sigma)\thh^k_i\thh^l_k\tg_{lj}+\thh^l_i[-2(F-\sigma)u\thh_{lj}]\\
&=[(F-\sigma)u]_{ij}-u(F-\sigma)\thh^k_i\thh_{kj}.
\end{align*}
Finally, by differentiating (\ref{Foh2}) with respect to $t$,
we get
\be\label{Evo1}
\begin{aligned}
&\frac{\partial}{\partial t}h_{ij}=\frac{1}{u}\dot{\thh}_{ij}-\frac{\thij}{u^2}u_t+\frac{\tgij}{u^2}\dot{\nu}^{n+1}
+\frac{\nu^{n+1}}{u^2}\dot{\tg}_{ij}-2\frac{\nu^{n+1}\tgij}{u^3}u_t\\
&=\frac{1}{u}\{[(F-\sigma)u]_{ij}-u(F-\sigma)\thh^k_i\thh_{kj}\}-\frac{\thij}{u}w(F-\sigma)\\
&+\frac{\tgij}{u^2}\{-\tg^{kl}[u(F-\sigma)]_ku_l\}+\frac{\nu^{n+1}}{u^2}[-2(F-\sigma)u\thij]
-2\frac{\nu^{n+1}\tgij}{u^3}uw(F-\sigma)\\
&=\frac{1}{u}\{[(F-\sigma)u]_{ij}-u(F-\sigma)\thh^k_i\thh_{kj}\}-\frac{\thij}{u}w(F-\sigma)\\
&-\{\tgukl[u(F-\sigma)]_ku_l\}\frac{\tgij}{u^2}-2\frac{(F-\sigma)\nu^{n+1}}{u}\thij-2\frac{\tgij}{u^2}(F-\sigma).
\end{aligned}
\ee
\end{proof}

\begin{lemma}\label{Evol3}(Evolution of F). The term $F$ evolves according to the equation
\be\label{Evo14}\begin{aligned}&F_t=\frac{1}{u}F^{ij}[(F-\sigma)u]_{ij}+(F-\sigma)\lll[\sum f_s\kappa^2_s-2\nu^{n+1}F+(\nu^{n+1})^2\sum f_s\rrr]\\
&+w(F-\sigma)\lll(F-\nu^{n+1}\sum f_s\rrr)-[(F-\sigma)u]_iu^i\sum f_s.\end{aligned}\ee
\end{lemma}
\begin{proof}
 We consider $F$ with respect to the mixed tensor $h^j_i.$ From equation (\ref{Foh2}), (\ref{Evo10}), and (\ref{Evo11}) we conclude
\be\label{Evo2}
\begin{aligned}
&F_t=F^{ij}(h^j_i)_t=F^{ij}\lll(u\thh^j_i+\nu^{n+1}\delta_{ij}\rrr)_t\\
&=uF^{ij}[(F-\sigma)u]^j_i+u^2(F-\sigma)F^{ij}\thh^k_i\thh^j_k\\
&+uw(F-\sigma)F^{ij}\thh^j_i-[(F-\sigma)u]_iu^i\sum f_s\\
&=\frac{1}{u}F^{ij}[(F-\sigma)u]_{ij}+(F-\sigma)\lll[\sum f_s\kappa_s^2-2\nu^{n+1}F+(\nu^{n+1})^2\sum f_s\rrr]\\
&+w(F-\sigma)\lll(F-\nu^{n+1}\sum f_s\rrr)-[(F-\sigma)u]_iu^i\sum f_s.
\end{aligned}
\ee
\end{proof}

\section{Gradient estimates} \label{Gre}
In this section we will show that the angle between the upward unit normal and $e_{n+1}$ axis
is bounded above upon approaching the boundary. We will also prove Proposition \ref{Grep0} which
gives us a global gradient bound for the solution.

The following lemma is similar to Theorem 3.1 of \cite{GS10}.
\begin{lemma}\label{Grel0}
For $\epsilon>0$ sufficiently small,
\be\label{Gre11}
\frac{\sigma-\nu^{n+1}}{u}<\frac{\sqrt{1-\sigma^2}}{r_1}+\frac{\epsilon(1+\sigma)}{r_1^2}\,on\,\partial\Omega\times[0,T),
\ee
where $r_1$ is the maximal radii of exterior sphere to $\partial\Omega.$
\end{lemma}
\begin{proof}
Applying Theorem \ref{Sel0} and letting $T$ be small enough, we first assume $r_1<\infty.$ For a fixed point $x_0\in\Gamma^{\epsilon},$ let $\bf{e}_1$ be the outward unit normal vector to $\Gamma^{\epsilon}$ at $x_0.$ Let $B_1$ be the ball in $\mathbb{R}^{n+1}$ of radius $R_1$ centered at $a=(x_0+r_1\mathbf{e}_1, R_1\sigma)$ where $R_1$ satisfies $R_1^2=r_1^2+(\epsilon-R_1\sigma)^2.$

Note that $B_1\cap P(\epsilon)=\{x\in\mathbb{R}^{n+1}|x_{n+1}=\epsilon\}$ is an n-ball of radius $r_1,$ which is externally tangent to $\Gamma^{\epsilon}.$ By Lemma 3.3 of \cite{LX10}, we know that $B_1\cap\Sigma(t)={\emptyset}$ for any $t\in[0,T).$ Hence, at $x_0$ we have
\[\nu^{n+1}>-\frac{u-\sigma R_1}{R_1}.\]
By an easy computation we also know that,
\[R_1\geq\frac{r_1^2}{\sqrt{(1-\sigma^2)r_1^2}+(1+\sigma)\epsilon}.\]
Therefore (\ref{Gre11}) is proved.
In the case that $r_1=\infty,$ then in the above argument one can replace $r_1$ by any $r>0$ and let $r\rightarrow\infty.$
\end{proof}

Now consider the approximation problem
\be\label{Gre0}
\left\{
\begin{aligned}
&G(D^2u,Du,u,u_t)=\frac{1}{uw}u_{t}-F=-\sigma\;&\mbox{in $\Omega_T$},\\
&u(x,t)=\epsilon\;&\mbox{on $\partial\Omega_T$},\\
&u(x,0)=u_0+\epsilon\;&\mbox{in $\Omega\times\{0\}$}.
\end{aligned}
\right.
\ee

By Lemma \ref{Grel0} we obtain a boundary gradient estimate
\be\label{Gre1}
|Du(x,t)|\leq C\;\mbox{on $\partial\Omega_T.$}
\ee

Similar to Lemma 5.1 of \cite{LX11}, we have
\begin{lemma}\label{Grel1}
If the initial surface $\Sigma_0$ satisfies $f(\Sigma_0)\leq\sigma,$ then $f(\Sigma(t))\leq\sigma,\;\forall (x,t)\in\Omega\times(0,T).$
\end{lemma}
\begin{proof}
By Lemma \ref{Evol3} we have
\be\label{Gre3}
\begin{aligned}
&\frac{\partial F}{\partial t}-F^{ij}\nabla_{ij}F\\
&=(F-\sigma)\left[\sum f_s\kappa^2_s-\nu^{n+1}F+(\nu^{n+1})^2\sum f_s+wF-2\sum f_s\right].
\end{aligned}
\ee
Now consider the function $\tilde{F}=e^{-\lambda t}(F-\sigma),$
\be\label{Gre4}
\begin{aligned}
&\frac{\partial\tilde{F}}{\partial t}-F^{ij}\nabla_{ij}\tilde{F}\\
&=\tilde{F}\left[\sum f_s\kappa^2_s-\nu^{n+1}F+(\nu^{n+1})^2\sum f_s+wF-2\sum f_s-\lambda\right].
\end{aligned}
\ee
If $\tilde{F}$ achieved its positive maximum at an interior point $(x_0,t_0)\in\Omega_T,$ then at this point we would have
$$\begin{aligned}
&\frac{\partial\tilde{F}}{\partial t}-F^{ij}\nabla_{ij}\tilde{F}\\
&=\tilde{F}\left[\sum f_s\kappa^2_s-\nu^{n+1}F+(\nu^{n+1})^2\sum f_s+wF-2\sum f_s-\lambda\right]\geq 0.
\end{aligned}$$
Choosing $\lambda$ big enough leads to a contradiction.
\end{proof}
Next we shall appeal to Theorem \ref{Sel0} and show that the linearized operator $\mathfrak{L}$ satisfies the maximum principle.
Moreover, we can get a uniform $C^1$ estimate for the admissible solution.

\begin{proposition}\label{Grep0}
Let $u(x,t)\in C^{2,1}(\Omega_{T})$ be an admissible solution of equation (\ref{Gre0}).
Suppose the initial surface $\Sigma_0=\{(x,u_0(x)),\,x\in\Omega\},$ satisfies $F(\Sigma_0)\leq\sigma$
and $\frac{1}{w(u_0(x))}>\sigma.$ Moreover, suppose $\mathcal{H}_{\partial\Omega}\geq 0,$ then $G_u\geq 0,$ $F(\Sigma(t))\leq\sigma$
and $w\leq\frac{1}{\sigma}$ on $\ol{\Omega}_{T}.$
\end{proposition}
\begin{proof}
From
\be\label{Gre2}
\begin{aligned}
G_u&=\frac{\partial G}{\partial u}=\frac{1}{u}(-2F+\sigma+\frac{1}{w}\sum f_i)\\
&\geq\frac{1}{u}(-2F+\sigma+\frac{1}{w}),
\end{aligned}
\ee
and the hypotheses on $\Sigma_0$ we can see that $G_u|_{t=0}>0.$ Thus when $t$ close to $0$, the linearized operator $\mathfrak{L}$
 satisfies the maximum principle. But  $\mathfrak{L}u_k=0$ so each derivative $u_k$ achieves its maximum on $\partial\Omega_{t^*},$ where $0<t^*<T$ sufficiently small. In particular, $w$ obtains its maximum on $\partial\Omega_{t^*}.$ By assumption we know that $w(u_0)<\frac{1}{\sigma},$ so we only need to assume that $w$ achieves its maximum on $\partial\Omega\times(0,t^*).$

Let $(0,t_0)\in\partial\Omega\times(0,t^*)$ be the point where $w$ assumes its maximum. Choose coordinates $(x_1,\cdots, x_n)$ at $0$ with $x_n$ the inner normal direction for $\partial\Omega.$ Then at $(0,t_0),$ we have
$$u_\alpha=0,\;1\leq\alpha<n,\;u_n>0,\;u_{nn}\leq 0,$$
$$\sum u_{\alpha\alpha}=-u_n(n-1)\mathcal{H}_{\partial\Omega}\leq 0.$$
Moreover, by equation (\ref{Int12}), the hyperbolic mean curvature of graph$(u)\geq F.$ Therefore by implying Theorem \ref{Sel0} we have
$$\frac{n}{\epsilon}(\sigma-\frac{1}{w})
\leq\frac{1}{w}\left(\sum_{\alpha<n}u_{\alpha\alpha}+\frac{u_{nn}}{w^2}\right)\leq-(n-1)\frac{u_n}{w}\mathcal{H}_{\partial\Omega}\leq 0.$$
Hence $\frac{1}{w}\geq\sigma$ on $\partial\Omega\times(0, t^*).$

Applying Lemma \ref{Grel1} we know $F\leq\sigma$ for all $t\in[0,T).$ Thus $G_u\geq 0$ so $\mathcal{L}$ satisfies the maximum principle. Consequently, the estimates must continue to hold as we increase $t^*$ up to $T.$
\end{proof}

\section{$C^2$ boundary estimates} \label{C2b}
In this section, we establish boundary estimates for second spatial derivatives of the admissible solutions to the Dirichlet problem (\ref{Int3}). According to (\ref{Foh15}) we can rewrite equation (\ref{Int3}) as follows:
\be\label{C2b0}
\left\{
\begin{aligned}
&\frac{1}{uw}u_t-F\lll(\frac{1}{w}(\delta_{ij}+u\gamma^{is}u_{sr}\gamma^{rj})\rrr)=-\sigma\;&\mbox{in $\Omega_T$},\\
&u(x,t)=\epsilon\;&\mbox{on $\partial\Omega\times[0,T)$},\\
&u(x,0)=u_0+\epsilon\;&\mbox{in $\Omega\times\{0\}$}.
\end{aligned}
\right.
\ee
As before we denote
\be\label{C2b1}
G(D^2u,Du,u,u_t)=\frac{1}{uw}u_t-F.
\ee

\begin{theorem}\label{C2bt0}
Suppose $f$ satisfies equation(\ref{Int5})-(\ref{Int11}). If $\epsilon$ is sufficiently small,
\be\label{C2b22}
u|D^2u|\leq C\;\mbox{on $\partial\Omega\times[0,T)$},
\ee
where $C$ is independent of $\epsilon$ and $t.$
\end{theorem}
Before we begin our proof note that
\be\label{C2b2}
G^{sr}:=\frac{\partial G}{\partial u_{sr}}=-\frac{u}{w}F^{ij}\gamma^{ir}\gamma^{sj},
\ee
\be\label{C2b3}
G^{sr}u_{sr}=-F+\frac{1}{w}\sum F^{ii},
\ee
\be\label{C2b4}
\begin{aligned}
&G_u:=\frac{\partial G}{\partial u}=-\frac{1}{wu^2}u_t-\frac{1}{w}F^{ij}\gamma^{ik}u_{kl}\gamma^{lj}\\
&=-\frac{\F}{u}-F^{ij}\lll(\frac{a_{ij}}{u}-\frac{1}{uw}\delta_{ij}\rrr)\\
&=-\frac{2F}{u}+\frac{\sigma}{u}+\frac{1}{wu}\sum F^{ii},
\end{aligned}
\ee
\be\label{C2b5}
G^t:=\frac{\partial G}{\partial u_t}=\frac{1}{uw},
\ee
\be\label{C2b6}
\begin{aligned}
&G^s:=\frac{\partial G}{\partial u_s}\\
&=-\frac{u_tu_s}{uw^3}+\frac{u_s}{w^2}F+\frac{2}{w}F^{ij}a_{ik}\lll(\frac{wu_k\gamma^{sj}+u_j\gamma^{ks}}{1+w}\rrr)
-\frac{2}{w^2}F^{ij}u_i\gamma^{sj}\\
&=-\frac{\F}{w^2}u_s+\frac{u_s}{w^2}F+\frac{2}{w}F^{ij}a_{ik}\lll(\frac{wu_k\gamma^{sj}+u_j\gamma^{ks}}{1+w}\rrr)
-\frac{2}{w^2}F^{ij}u_i\gamma^{sj}\\
&=\frac{u_s}{w^2}\sigma+\frac{2}{w}F^{ij}a_{ik}\lll(\frac{wu_k\gamma^{sj}+u_j\gamma^{ks}}{1+w}\rrr)
-\frac{2}{w^2}F^{ij}u_i\gamma^{sj}.\\
\end{aligned}
\ee
Thus
\be\label{C2b7}
G^su_s=\frac{w^2-1}{w^2}\sigma+\frac{2}{w^2}F^{ij}a_{ik}u_ku_j-\frac{2}{w^3}F^{ij}u_iu_j
\ee
and
\be\label{C2b24}
\sum |G^s|\leq \frac{\sigma}{w}+\frac{2}{w}\lll(\sum F^{ii}+\sum f_i|\kappa_i|\rrr).
\ee

Now let $\mathcal{L}'$ denote the partial linearized operator of $G$ at $u$:
\[\mathcal{L}'=\mathcal{L}-G_u=G^t\partial_t+G^{sr}\partial_s\partial_r+G^s\partial_s.\]
By equation (\ref{C2b3}), (\ref{C2b5}) and (\ref{C2b7}) we get
\be\label{C2b8}
\begin{aligned}
\mathcal{L}'u&=G^tu_t+G^{sr}u_{sr}+G^su_s\\
&=\frac{1}{uw}uw(F-\sigma)-F+\frac{1}{w}\sum F^{ii}+\left(1-\frac{1}{w^2}\right)\sigma\frac{2}{w^2}F^{ij}a_{ik}u_ku_j-\frac{2}{w^3}F^{ij}u_iu_j\\
&=-\frac{1}{w^2}\sigma+\frac{1}{w}\sum F^{ii}+\frac{2}{w^2}F^{ij}a_{ik}u_ku_j-\frac{2}{w^3}F^{ij}u_iu_j,
\end{aligned}
\ee
hence
\be\label{C2b9}
\begin{aligned}
\mathcal{L}\frac{1}{u}&=G^t\lll(-\frac{u_t}{u^2}\rrr)+G^{sr}\partial_s\lll(-\frac{u_r}{u^2}\rrr)+G^s\lll(-\frac{u_s}{u^2}\rrr)+G_u\lll(\frac{1}{u}\rrr)\\
&=-\frac{1}{u^2}\lll(G^tu_t+G^{sr}u_{sr}+G^su_s\rrr)+2G^{sr}\frac{u_ru_s}{u^3}+G_u\frac{1}{u}\\
&=-\frac{1}{u^2}\lll(-\frac{1}{w^2}\sigma+\frac{1}{w}\sum F^{ii}+\frac{2}{w^2}F^{ij}a_{ik}u_ku_j-\frac{2}{w^3}F^{ij}u_iu_j\rrr)\\
&+2G^{sr}\frac{u_ru_s}{u^3}+\frac{1}{u}\lll(-2\frac{F}{u}+\frac{\sigma}{u}+\frac{1}{uw}\sum F^{ii}\rrr)\\
&=\frac{1+w^2}{u^2w^2}\sigma-2\frac{F}{u^2}-\frac{2}{u^2w^2}F^{ij}a_{ik}u_ku_j.
\end{aligned}
\ee

\begin{lemma}\label{C2bl0}
Suppose that $f$ satisfies (\ref{Int5}), (\ref{Int6}), (\ref{Int9}) and (\ref{Int10}). Then
\be\label{C2b10}
\mathcal{L}\lll(1-\frac{\epsilon}{u}\rrr)\geq\frac{\epsilon}{u^2w}(1-\sigma)\sum f_i+\frac{2\epsilon}{u^2w^2}F^{ij}a_{ik}u_ku_j\;\mbox{in $\Omega_T.$}
\ee
\end{lemma}
\begin{proof}
By equation (\ref{C2b9}) and Proposition \ref{Grep0}
\be\label{C2b11}
\begin{aligned}
&\mathcal{L}\lll(1-\frac{\epsilon}{u}\rrr)=\mathcal{L}'\lll(1-\frac{\epsilon}{u}\rrr)+G_u\lll(1-\frac{\epsilon}{u}\rrr)\\
&\geq\mathcal{L}'\lll(1-\frac{\epsilon}{u}\rrr)=-\epsilon\mathcal{L}'\lll(\frac{1}{u}\rrr)=-\epsilon\lll(\mathcal{L}-G_u\rrr)\frac{1}{u}\\
&\geq\frac{\epsilon}{u^2w}(1-\sigma)\sum f_i+\frac{2\epsilon}{u^2w^2}F^{ij}a_{ik}u_ku_j.
\end{aligned}
\ee
\end{proof}

Recall that for symmetric matrix $A=A[u]$ we can uniquely define the symmetric matrices
$$|A|=\{AA^T\}^{\frac{1}{2}},\;A^+=\frac{1}{2}(|A|+A),\;A^-=\frac{1}{2}(|A|-A)$$
which all commute and satisfy $A^+A^-=0.$ Moreover, $F=F^{ij}$ commutes with $|A|,$ $A^{\pm}$ so all simultaneously diagonalizable.
Write $A^{\pm}=\{a_{ij}^{\pm}\}$ and define
\be\label{C2b21}
L=\mathcal{L}+\frac{2}{w^2}F^{ij}a_{ik}^-u_k\partial_j.
\ee

\begin{corollary}\label{C2bc0}
Suppose that $f$ satisfies (\ref{Int5}), (\ref{Int6}), (\ref{Int9}) and (\ref{Int10}). Then
\be\label{C2b12}
L\lll(1-\frac{\epsilon}{u}\rrr)\geq\frac{\epsilon(1-\sigma)}{u^2w}\sum f_i.
\ee
\end{corollary}

Finally we need to point out that, similar to \cite{CNS84} we can prove

\begin{lemma}\label{C2bl2}
Suppose that $f$ satisfies (\ref{Int5}), (\ref{Int6}), (\ref{Int9}) and (\ref{Int10}). Then
\be\label{C2b25}
\mathcal{L}(x_iu_j-x_ju_i)=0,\;\mathcal{L}u_i=0,\;\;1\leq i,j\leq n.
\ee
\end{lemma}

\begin{proof}[Proof of Theorem \ref{C2bt0}] Consider an arbitrary point on $\partial\Omega,$ which we may assume to be the origin of $\mathbb{R}^n,$ and choose the coordinates so that the positive $x_n$ axis is the interior normal to $\partial\Omega$ at the origin. There exists a uniform constant $r>0$ such that $\partial\Omega\cap B_r(0)$ can be represented as a graph
\[x_n=\frac{1}{2}\sum_{\alpha,\beta <n}B_{\alpha\beta}x_{\alpha}x_{\beta}+O(|x'|^3)=\rho(x'),\;\;x'=(x_1, \cdots, x_{n-1}).\]
Since $u\equiv\epsilon,\;\mbox{on $\partial\Omega\times[0,T),$}$ i.e., $u(x',\rho(x'))\equiv\epsilon$ for $\forall t\in[0,T),$
we have at the origin that
\[u_{\alpha}+u_nB_{\alpha\beta}x_{\beta}=0\,\,,\,\,u_{\alpha\beta}+u_n\rho_{\alpha\beta}=0,\,\,\forall t\in[0,T)\mbox{ and $\alpha,$ $\beta<n$ }.\]
As in \cite{CNS84}, let $T_{\alpha}=\partial_{\alpha}+\sum_{\beta<n}B_{\alpha\beta}(x_{\beta}\partial_{n}-x_n\partial_{\beta}).$
For fixed $\alpha<n,$ we have
\be\label{C2b14}
|T_{\alpha}u|\leq C \;\;\mbox{in $\{\Omega\cap B_{\epsilon}(0)\}\times[0,T),$}
\ee
\be\label{C2b13}
|T_\alpha u|\leq C|x|^2 \;\;\mbox{on $\{\partial\Omega\cap B_{\epsilon}(0)\}\times[0, T),$}
\ee
where $C$ is independent of $\epsilon$ and $T.$
Moreover by Lemma \ref{C2bl2}
\be\label{C2b15}
\mathcal{L}T_{\alpha}u=0.
\ee

Now define $$\phi=\pm T_{\alpha}u+\frac{1}{2}\sum_{l<n}u_l^2-\frac{C}{\epsilon^2}|x|^2,$$ where $C$ is chosen to be large enough (and independent of $\epsilon$ and $T$) so that
$\phi\leq 0$ on $\partial(\Omega\cap B_{\epsilon}(0))\times[0,T).$ Since $u_0\in C^2(\Omega)$ is given,
from Taylor's theorem we can assume in $\Omega\times B_{\delta}(0),$ $\delta>\epsilon>0$ is small, there exists $a_1, b_1, b_2$ and $c_1>0$ so that
\[u_0(x)\geq \epsilon +a_1x_n,\;\;\lll|T_{\alpha}u_0\rrr|\leq b_1x_n+b_2|x|^2 \;\mbox{and}\; |u_{0l}|\leq c_1|x|.\]
Therefore, we can choose a constant $C_1>0$ such that
\be\label{C2b31}
\phi- C_1\lll(1-\frac{\epsilon}{u}\rrr)\leq 0\;\;\mbox{on $\{\Omega\cap B_\epsilon(0)\}\times\{0\}$},
\ee
here and in the future, all $C$ and $C_i$ ($i=1,2,\cdots$) denote constants independent of $\epsilon$ and $t.$

\begin{lemma}\label{C2bl1}
\be\label{C2b16}
\mathcal{L}\phi\leq\sum_{l<n}G^{sr}u_{ls}u_{lr}+\frac{C}{\epsilon}\lll(\sum f_i+\sum f_i|\kappa_i|\rrr)\;\mbox{in $\{\Omega\cap B_{\epsilon}(0)\}\times[0,T)$.}
\ee
\end{lemma}
\begin{proof}
Since
\be\label{C2b17}
\begin{aligned}
&\mathcal{L}(|x|^2)=G^{sr}\partial_s\partial_r|x|^2+G^s\partial_s|x|^2+G_u|x|^2\\
&\leq\lll|2\sum G^{ss}+2\sum x_sG^s+|x|^2G_u\rrr|\\
&\leq 2\lll|\sum G^{ss}\rrr|+2\epsilon|G^s|+\epsilon^2|G_u|\\
&\leq\frac{2C\epsilon}{w}\sum f_i+2\epsilon\lll(\frac{\sigma}{w}+\frac{2}{w}(\sum f_i+\sum f_i|\kappa_i|)\rrr)
+C\epsilon\lll(\sum f_i+\sum f_i|\kappa_i|\rrr)
\end{aligned}
\ee
where we applied lemma 2.1 of \cite{GS08} and Lemma 3.3 of \cite{LX10}.

Combining (\ref{C2b17}) with Lemma \ref{C2bl2} we obtain (\ref{C2b16}).
\end{proof}

Following Ivochkina, Lin and Trudinger \cite{ILT96} we have
\begin{proposition}\label{C2bp0}
At each point in $\{\Omega\cap B_\epsilon(0)\}\times[0,T)$ there is an index $r$ such that
\be\label{C2b18}
\sum_{l<n}G^{sr}u_{ls}u_{lr}\leq-c_0u\sum_{i\neq r}f_i\tilde{\kappa}^2_i\leq\frac{c_0}{2u}\lll(\frac{2}{w^2}\sum f_i-\sum_{i\neq r}f_i\kappa_i^2\rrr).
\ee
\end{proposition}
\begin{proof}
Let $P$ be an orthogonal matrix that simultaneously diagonalizes $\{F^{ij}\}$ and $\tilde{A}=\{\tilde{a}_{ij}\}=\{\frac{1}{w}\gamma^{ik}u_{kl}\gamma^{lj}\},$
where $\gamma^{ij}=\delta_{ij}-\frac{u_iu_j}{w(1+w)}.$

Note that $w\tilde{a}_{ij}\gamma_{jl}=\gamma^{ik}u_{kl}$ and so we have
\be\label{C2b19}
\begin{aligned}
&\sum_{l<n}G^{sr}u_{ls}u_{lr}=-uwF^{ij}\tilde{a}_{iq}\tilde{a}_{jp}\gamma_{ql}\gamma_{pl}\\
&=-uw\sum_{l<n}f_i\tilde{\kappa}_i^2P_{pi}\gamma_{pl}P_{qi}\gamma_{ql}\\
&=-uw\sum_{l<n}f_i\tilde{\kappa}_i^2b^2_{li},
\end{aligned}
\ee
where $B=\{b_{rs}\}=\{P_{ir}\gamma_{is}\}$ and $\det(B)=\det(B^T)=w.$

Suppose for some $i,$ say $i=1,$ we have
$$\sum_{l<n}b_{l1}^2<\theta^2.$$
Expanding $\det B$ by cofactors along the first column gives
$$1\leq w=\det B=b_{11}C^{11}+\cdots+b_{n-11}C^{1n-1}+b_{n1}\det M\leq c_1\theta+c_2\det M,$$
where $c_1$,$c_2$ are independent of $\epsilon$ and $T,$ and
\[M= \left[ \begin{array}{ccc}
b_{12} & \cdots & b_{n-12} \\
\vdots& \ddots & \vdots \\
b_{1n}& \cdots & b_{n-1n} \end{array} \right].\]
Therefore, $\det M\geq\frac{1-c_1\theta}{c_2}.$ Now expanding $\det M$ by cofactor along row $r\geq2$ gives
$\det M\leq c_3\lll(\sum_{l<n}b_{lr}^2\rrr)^{1/2},$ where $c_3$ is independent of $\epsilon$ and $T.$ Hence
\be\label{C2b20}
\sum_{l<n}b^2_{lr}\geq\lll(\frac{1-c_1\theta}{c_2c_3}\rrr)^2.
\ee
Choosing $\theta<\frac{1}{2c_1}$ we conclude that for some $r$
\[\sum_{l<n}G^{sr}u_{ls}u_{lr}\leq-c_0u\sum_{i\neq r}f_i\tilde{\kappa}_i^2.\]
Finally (\ref{C2b18}) follows from equation (\ref{Foh3}).
\end{proof}

\begin{proposition}\label{C2bp1}
Let $L$ be defined by (\ref{C2b21}). Then
\be\label{C2b23}
L\phi\leq C_2\lll(\frac{1}{\epsilon}\sum f_i-G^{ij}\phi_i\phi_j\rrr)
\ee
for a controlled constant $C_2$ independent of $\epsilon$ and $t.$
\end{proposition}
\begin{proof}
By Lemma \ref{C2bl1} and Proposition \ref{C2bp0},
\be\label{C2b26}
\begin{aligned}
&L\phi=\mathcal{L}\phi+\frac{2}{w^2}F^{ij}a^{-}_{ik}u_k\partial_j\phi\\
&\leq\sum_{l<n}G^{sr}u_{ls}u_{lr}+\frac{C}{\epsilon}(\sum f_i+\sum f_i|\kappa_i|)+\frac{2}{w^2}F^{ij}a^-_{ik}u_k\phi_j\\
&\leq\frac{c_0}{uw^2}\sum f_i-\frac{c_0}{2u}\sum_{i\neq r}f_i\kappa_i^2+\frac{C}{\epsilon}(\sum f_i+\sum f_i|\kappa_i|)+\frac{2}{w^2}F^{ij}a^-_{ik}u_k\phi_j.
\end{aligned}
\ee
Implying the generalized Schwarz inequality,
\be\label{C2b27}
\begin{aligned}
\frac{2}{w^2}\lll|F^{ij}a^-_{ik}u_k\phi_j\rrr|&\leq2\lll(uF^{ij}\phi_i\phi_j\rrr)^{\frac{1}{2}}
\lll(\frac{1}{u}F^{ij}a^-_{il}a^-_{kj}\frac{u_ku_l}{w^2}\rrr)^{\frac{1}{2}}\\
&\leq\frac{c_0}{8nu}\sum_{\kappa_i<0}f_i\kappa^2_i-CG^{ij}\phi_i\phi_j,
\end{aligned}
\ee
where we have used Lemma 2.1 of \cite{GS08} to compare $uF^{ij}\phi_i\phi_j$ to $-G^{ij}\phi_i\phi_j.$
Moreover,
\be\label{C2b28}
\sum f_i|\kappa_i|=\sum_{\kappa_i>0}f_i\kappa_i-\sum_{\kappa_i<0}f_i\kappa_i=F+2\sum_{\kappa_i<0}f_i|\kappa_i|.
\ee
Hence we get equation (\ref{C2b23}) with $C_2$ independent of $\epsilon$ and $t.$
\end{proof}

Let $h=\lll(e^{C_2\phi}-1\rrr)-A\lll(1-\frac{\epsilon}{u}\rrr),$ with $C_2$ defined as before and $A$ to be determined later. From equation (\ref{C2b31}) we know that when $A$ is chosen large enough
\be\label{C2b29}
h\leq 0\;\;\mbox{on $\partial\{(\Omega\cap B_\epsilon(0))\times[0,T)\}$.}
\ee
Moreover, by Proposition \ref{C2bp1} and Corollary \ref{C2bc0} we get
\be\label{C2b30}
Lh\leq 0\;\;\mbox{in $(\Omega\cap B_\epsilon(0))\times[0,T)$.}
\ee
Therefore by the maximum principle we conclude that $h\leq 0$ in $(\Omega\cap B_\epsilon(0))\times[0,T).$ Since $h(0,t)=0,$ we have that $h_n(0,t)\leq 0$ for all $t\in[0,T)$
which gives
\be\label{C2b32}
|u_{\alpha n}(0,t)|\leq\frac{A}{C_2\epsilon}u_n(0,t)\;\;\mbox{for all $t\in[0,T).$}
\ee

Finally, $|u_{nn}(0,t)|$ can be estimated as in [L.Xiao] section 6 using the hypothesis (\ref{Int11}). For completeness we include the argument here. For any $t\in[0,T),$ we may assume $[u_{\alpha\beta}(0,t)]$ to be diagonal. Note also that $u_{\alpha}(0,t)=0$ for $\alpha<n.$ We have at $(x,t)=(0,t)$
$$A[u]=\frac{1}{w}
\lll[\begin{array}{cccc}
1+uu_{11}&0&\cdots&\frac{uu_{1n}}{w}\\
0&1+uu_{22}&\cdots&\frac{uu_{2n}}{w}\\
\vdots&\vdots&\ddots&\vdots\\
\frac{uu_{n1}}{w}&\frac{uu_{n2}}{w}&\cdots&1+\frac{uu_{nn}}{w^2}\\
\end{array}\rrr]. $$

By lemma 1.2 in \cite{CNS85}, if $\epsilon u_{nn}(0)$ is very large, the eigenvalues $\lambda_1,\cdots,\lambda_n$
of $A[u]$ are given by
\be\label{C2b33}
\begin{aligned}
&\lambda_\alpha=\frac{1}{w}(1+\epsilon u_{\alpha\alpha}(0))+o(1),\,\alpha<n\\
&\lambda_n=\frac{\epsilon u_{nn}(0)}{w^3}\lll(1+O\lll(\frac{1}{\epsilon u_{nn}(0)}\rrr)\rrr).
\end{aligned}
\ee
If $\epsilon u_{nn}\geq R$ where $R$ is a uniform constant, then by (\ref{Int10}), (\ref{Int11}) and Proposition \ref{Grep0} we have
\[\sigma\geq\frac{1}{w}F(wA[u])(0)\geq(\sigma-C\epsilon)\lll(1+\frac{\epsilon_0}{2}\rrr)>\sigma\]
which is a contradiction. Therefore
\[|u_{nn}(0)|\leq\frac{R}{\epsilon}\]
and the proof is completed.
\end{proof}

\section{$C^2$ global estimates}\label{c2g}

In this section we will prove a maximum principle for the largest hyperbolic principal curvature $\kappa_{\max}(x, t)$ of solutions of $f(\kappa[u(x,t)])=\sigma.$

As before, we denote the metric in $\mathbb{H}^{n+1}$ by $g_{ij}$ and denote the hyperbolic second fundamental form by $h_{ij}.$
Now consider function
\be\label{C2g3}
\varphi=\max_{(x,t)\in\overline{\Omega}_T}\frac{\kappa_{\max}(x,t)}{\nu^{n+1}-a},
\ee
where $\inf_{\overline{\Omega}_T}\nu^{n+1}>a.$
\begin{theorem}\label{c2gt0}
Suppose $f$ satisfies (\ref{Int5})-(\ref{Int10}) and $\sigma\in(0,1)$ satisfies $\sigma>\sigma_0,$
where $\sigma_0$ is the unique zero in $(0,1)$ of
\be\label{c2g40}
\phi(a):=\frac{4}{3}a-\frac{1}{27}a^3-\frac{1}{27}(a^2+3)^{\frac{3}{2}}.
\ee
Let $u\in C^{4,2}(\Omega\times[0,T))$ be an admissible solution of (\ref{Int3}) such that $\nu^{n+1}(x,t)=\frac{1}{w}\geq\sigma,$ for all $(x,t)\in\Omega_T.$
Then at an interior maximum of $\varphi,$ there is a constant C (independent of $\epsilon$ and $t$), such that
\be\label{c2g41}
\kappa_{\max}\leq C.
\ee
Numerical calculations show $0.14596<\sigma_0<0.14597.$
\end{theorem}

We begin the proof of Theorem \ref{c2gt0} which is long and computational.

Assume $\varphi$ achieves its maximum at an interior point $(x_0,t_0).$ We may rewrite $\Sigma(t_0)$ locally near $\mathbf{X}_0=(x_0, u(x_0,t_0))$ as a radial graph $\mathbf{X}=e^{v(\mathbf{z},t)}\mathbf{z},\;(\mathbf{z},t)\in\mathbb{S}^n_+\times(0,T),$ such that $\nu(\mathbf{X}_0)=\mathbf{z}_0,$ and we may also choose the local
coordinates around $\mathbf{z}_0\in\mathbb{S}^n_+$ such that at $(\mathbf{z}_0,t_0)$
$$\tilde{g}_{ij}=\delta_{ij}\;\;\mbox{and}\;\;\frac{\partial\tg_{ij}}{\partial\theta^k}=0.$$
By a standard calculation, we also know that $v(\mathbf{z}, t)$ satisfies
$$v_t=yw(f-\sigma).$$
Moreover, we can also assume $\thh_{ij}$ is diagnalized at $(\mathbf{z}_0,t_0).$
At last, since dilation is an isometry for radial graph, without loss of generality we may assume $v(\mathbf{z}_0,t_0)=0$
Therefore at $(\mathbf{z}_0,t_0)$ we have
\be\label{C2g4}
g_{ij}=\frac{\tg_{ij}}{u^2}=\frac{\delta_{ij}}{y^2}\;\;and\;\;g^{ij}=u^2\tg^{ij}=y^2\delta_{ij},
\ee
\be\label{C2g5}
h_{ij}=\frac{1}{u}\thh_{ij}+\frac{\nu^{n+1}}{u^2}\tg_{ij}=\frac{\thh_{ij}}{y}+\frac{\delta_{ij}}{y}.
\ee
Differentiating (\ref{Foh17}) with respect to $\theta^k$ we get
\be\label{C2g6}
\begin{aligned}
&\frac{\partial\tg_{ij}}{\partial\theta^k}=\frac{\partial[e^{2v}(\sigma_{ij}+v_iv_j)]}{\partial\theta^k}\\
&=2e^{2v}v_k(\sigma_{ij}+v_iv_j)+e^{2v}\lll(\frac{\partial\sigma_{ij}}{\partial\theta^k}+v_{ik}v_j+v_iv_{jk}\rrr)=0.
\end{aligned}
\ee
Since $\nu(\mathbf{X}_0)=\mathbf{z}_0,$ we conclude that at $(\mathbf{z}_0,t_0)$
\[\frac{\partial\sigma_{ij}}{\partial\theta^k}=0\]
which implies
$$\Gamma'^k_{ij}=0.$$
Thus
\be\label{C2g7}
\nabla'_{ij}v=v_{ij}=\tilde{\nabla}_{ij}v,
\ee
where $\tilde{\nabla}_{ij}$ denotes the covariant differentiations in the metric $\tg$ with respect to the local coordinates on $\Sigma(t_0).$

Recall that by Lemma \ref{Evol0} we have
$$\frac{\partial\tg_{ij}}{\partial t}=-2(F-\sigma)u\thh_{ij}.$$
On the other hand
\be\label{C2g8}
\begin{aligned}
\frac{\partial\tg_{ij}}{\partial t}&=2e^{2v}v_t(\sigma_{ij}+v_iv_j)+e^{2v}\lll(\dot{\sigma}_{ij}+\dot{v}_iv_j+v_i\dot{v}_j\rrr)\\
&=2\tg_{ij}yw(F-\sigma)+e^{2v}\lll(\dot{\sigma}_{ij}+\dot{v}_iv_j+v_i\dot{v}_j\rrr).
\end{aligned}
\ee
Therefore at $(\mathbf{z}_0,t_0)$
\be\label{C2g9}
\dot{\sigma}_{ij}=-2y(F-\sigma)\thh_{ij}-2y(F-\sigma)\delta_{ij}.
\ee
Combining equation (\ref{Foh19}) and (\ref{C2g9}) we get
\be\label{C2g10}
\begin{aligned}
\frac{\partial{\thh_{ij}}}{\partial t}&=\thh_{ij}y(F-\sigma)+\nabla'_{ij}[yw(F-\sigma)]+2(F-\sigma)y\thh_{ij}+2y(F-\sigma)\delta_{ij}\\
&=3\thh_{ij}y(F-\sigma)+\nabla'_{ij}[yw(F-\sigma)]+2y(F-\sigma)\delta_{ij}\\
&=3\thh_{ij}y(F-\sigma)+\{y\nabla'_{ij}F+y(F-\sigma)v_{li}v_{lj}\\
&-(F-\sigma)y\delta_{ij}+y_iF_j+y_jF_i\}+2y(F-\sigma)\delta_{ij}\\
&=3\thh_{ij}y(F-\sigma)+y\nabla'_{ij}F+y(F-\sigma)v_{li}v_{lj}+y_iF_j+y_jF_i+y(F-\sigma)\delta_{ij}.
\end{aligned}
\ee

We can always assume at $(\mathbf{z}_0, t_0)$ $\kappa_{\max}=g^{11}h_{11},$ thus we only need to compute $\dot{h}_{11}$ at this point. From now on, all calculations are done at $(\mathbf{z}_0,t_0)$
if no additional explanations.

\begin{lemma}\label{C2glm0}
At $(\mathbf{z}_0, t_0),$
\be\label{C2g25}
\begin{aligned}
\frac{\partial}{\partial t}h^1_1-y^2F^{ii}\nabla_{ii}h^1_1&=3\F\kappa_1^2+y^2F^{ij,kl}h^j_{i;1}h^l_{k;1}\\
&-\F+\lll(\kappa_1\sum f_i\kappa_i^2+\kappa_1\sum f_i-F-\kappa_1^2F\rrr).
\end{aligned}
\ee
\end{lemma}
\begin{proof}Differentiating equation (\ref{C2g5}) with respect to $t$ we get
\be\label{C2g11}
\dot{h}_{ij}=\frac{1}{u}\dot{\thh}_{ij}-\frac{\thh_{ij}}{u^2}\dot{u}+\frac{\dot{\nu^{n+1}}}{u^2}\tg_{ij}
+\frac{\nu^{n+1}}{u^2}\dot{\tg}_{ij}-2\frac{\nu^{n+1}}{u^3}\tg_{ij}\dot{u}.
\ee
Since
\be\label{C2g12}
\dot{u}=\frac{\partial e^vy}{\partial t}=e^vyv_t=y^2(F-\sigma)
\ee
and
\be\label{C2g13}
\begin{aligned}
\dot{\nu}^{n+1}&=\frac{\partial}{\partial t}\left\{\frac{y-\nabla'v\cdot\nabla'y}{w}\right\}\\
&=-\nabla'[yw(F-\sigma)]\cdot\nabla'y\\
&=-(F-\sigma)(1-y^2)-y\nabla'F\cdot\nabla'y,
\end{aligned}
\ee
we obtain
\be\label{C2g14}
\begin{aligned}
\dot{h}_{11}&=\frac{1}{y}\lll[3\thh_{11}y(F-\sigma)+y\nabla'_{11}F+y(F-\sigma)v_{l1}v_{l1}+2y_1F_1+y(F-\sigma)\rrr]\\
&-\thh_{11}(F-\sigma)-\frac{(F-\sigma)(1-y^2)}{y^2}-\frac{y\nabla'F\cdot\nabla'y}{y^2}\\
&-2(F-\sigma)\thh_{11}-2(F-\sigma)\\
&=\nabla'_{11}F+(F-\sigma)v_{l1}^2+\frac{2}{y}y_1F_1-\frac{(F-\sigma)}{y^2}-\frac{1}{y}\nabla'F\cdot\nabla'y\\
&=\nabla'_{11}F+(F-\sigma)(\thh_{11}+1)^2+\frac{2}{y}y_1F_1-\frac{(F-\sigma)}{y^2}-\frac{1}{y}\nabla'F\cdot\nabla'y.
\end{aligned}
\ee
Here we used $\thh_{ij}=v_{ij}-\delta_{ij}$ at $(\mathbf{z}_0, t_0).$

By equation (\ref{Foh0}), (\ref{Evo7}) and (\ref{C2g12}) at the point $(\mathbf{z}_0,t_0)$ we get
\be\label{C2g15}
\dot{g}_{ij}=-2(F-\sigma)h_{ij}.
\ee
On the other hand
$$\frac{\partial}{\partial t}\lll(g^{ik}g_{ij}\rrr)=\frac{\partial}{\partial t}\delta^k_j=0,$$
hence
\be\label{C2g16}
\dot{g}^{ik}=2y^4(F-\sigma)h_{ki}.
\ee
Finally we have
\be\label{C2g17}
\begin{aligned}
\frac{\partial h^1_1}{\partial t}&=\frac{\partial}{\partial t}\lll(g^{1k}h_{k1}\rrr)=\dot{g}^{1k}h_{k1}+g^{1k}\dot{h}_{k1}\\
&=2y^4(F-\sigma)h_{k1}^2+y^2\dot{h}_{11}\\
&=2y^4(F-\sigma)h_{1k}^2+y^2\left[\nabla'_{11}F+(F-\sigma)(\thh_{11}+1)^2+\frac{2}{y}y_1F_1\right.\\
&\left.-\frac{\F}{y^2}-\frac{1}{y}\nabla'F\cdot\nabla'y\right].
\end{aligned}
\ee

By (\ref{Foh10}) at $(\mathbf{z}_0,t_0)$ we get
\be\label{C2g19}
\na_{ij}f=\tna_{ij}f+\frac{1}{y}\lll(y_if_j+y_jf_i-\sum y_lf_l\delta_{ij}\rrr).
\ee
Therefore we can rewrite equation (\ref{C2g17}) as
\be\label{C2g20}
\frac{\partial h^1_1}{\partial t}=3(F-\sigma)\kappa_1^2+y^2\na_{11}F-\F.
\ee
Moreover
\be\label{C2g21}
\begin{aligned}
\na_{11}F&=F^{ii}h^i_{i;11}+F^{ij,kl}h^j_{i;1}h^l_{k;1}\\
&=F^{ii}\na_{11}\lll(g^{ik}h_{ki}\rrr)+F^{ij,kl}h^j_{i;1}h^l_{k;1}\\
&=y^2F^{ii}\na_{11}h_{ii}+F^{ij,kl}h^j_{i;1}h^l_{k;1}.
\end{aligned}
\ee
Thus
\be\label{C2g22}
\frac{\partial h^1_1}{\partial t}=3\F\kappa^2_1+y^4F^{ii}\na_{11}h_{ii}+y^2F^{ij,kl}h^j_{i;1}h^l_{k;1}-\F.
\ee

Next let's recall the following well-known fundamental equations for a hypersurface $\Sigma$ in $\mathbb{H}^{n+1}:$
\[\begin{aligned}
&\mbox{Coddazzi equation:}\;\; \na_ih_{jk}=\na_jh_{ki}=\na_kh_{ij},\\
&\mbox{Gauss equation:}\;\;R_{ijkl}=(h_{ik}h_{jl}-h_{il}h_{jk})-(g_{ik}g_{jl}-g_{il}g_{jk}),\\
&\mbox{Ricci equation:}\;\;\na_l\na_kh_{ij}-\na_k\na_lh_{ij}=h_{ip}g^{pq}R_{qjkl}+h_{jp}g^{pq}R_{qikl}.
\end{aligned}\]
So we have
\be\label{C2g23}
\begin{aligned}
\na_{11}h_{ii}-\na_{ii}h_{11}&=\na_1\na_ih_{1i}-\na_i\na_1h_{1i}\\
&=h_{1p}g^{pq}R_{qii1}+h_{ip}g^{pq}R_{q1i1}\\
&=\kappa_1R_{1ii1}+\kappa_iR_{i1i1}
\end{aligned}
\ee
and
\be\label{C2g24}
R_{1ii1}=-h_{11}h_{ii}+\frac{1}{y^4},\;R_{i1i1}=h_{ii}h_{11}-\frac{1}{y^4}.
\ee
Substituting equation (\ref{C2g24}) into (\ref{C2g23}) and combining with equation (\ref{C2g22}) we obtain
(\ref{C2g25}).
\end{proof}
\begin{lemma}\label{C2glm1}
At $(\mathbf{z}_0, t_0),$
\be\label{C2g29}
\begin{aligned}
\dot{\nu}^{n+1}-y^2F^{ii}\nabla_{ii}\nu^{n+1}
&=-\F(1-y^2)-(1-y^2)\sum f_i(\kappa_i-y)\\&+y\sum f_i(\kappa_i-y^2)-2y\sum f_iy_i\nu_i^{n+1}.
\end{aligned}
\ee
\end{lemma}
\begin{proof}
By differentiating $\nu^{n+1}$ we get
\be\label{C2g26}
\begin{aligned}
\nu^{n+1}_i&=\lll(\frac{y-\na'v\cdot\na'y}{w}\rrr)_i\\
&=\frac{y_i-v_{li}y_l-v_ly_{li}}{w}-\frac{y-\na'v\cdot\na'y}{w^2}w_i\\
&=y_i-v_{li}y_l=-\thh_{ii}y_i\\
\end{aligned}
\ee
and
\be\label{C2g27}
\begin{aligned}
\nu^{n+1}_{ij}&=\frac{y_{ij}-v_{lij}y_l-v_{li}y_{lj}-v_{lj}y_{li}-v_ly_{lij}}{w}\\
&-\frac{y-\na'v\cdot\na'y}{w^2}w_{ij}-w_i\lll(\frac{y-\na'v\cdot\na'y}{w^2}\rrr)_j\\
&=-y\delta_{ij}-v_{lij}y_l+yv_{li}\delta_{lj}+yv_{lj}\delta_{li}-yv_{li}v_{lj}\\
&=-y_l\tna_l\thh_{ij}-y\thh_{li}\thh_{lj}.
\end{aligned}
\ee
Then
\be\label{C2g28}
\begin{aligned}
\na_{ij}\nu^{n+1}&=\tna_{ij}\nu^{n+1}+\frac{1}{y}\lll(y_i\nu^{n+1}_j+y_j\nu^{n+1}_i\rrr)-\frac{1}{y}\sum y_l\nu^{n+1}_l\delta_{ij}\\
&=-y_l\na_l\thh_{ij}-y\thh_{li}\thh_{lj}+\frac{1}{y}\lll(y_i\nu^{n+1}_j+y_j\nu^{n+1}_i\rrr)-\frac{1}{y}\sum y_l\nu^{n+1}_l\delta_{ij}.
\end{aligned}
\ee
Moreover, differentiating $F$ with respect to $\tau_l,$
\be\label{C2g42}
\begin{aligned}
\nabla_lF&=F^{ij}\lll(h_i^j\rrr)_l=F^{ij}\lll(u\thh_i^j+\nu^{n+1}\delta_{ij}\rrr)_l\\
&=F^{ij}\lll\{y_l\thh_{ij}+y\nabla_l\thh_{ij}+\nu^{n+1}_l\delta_{ij}\rrr\}.
\end{aligned}
\ee

Combining equations (\ref{C2g13}), (\ref{C2g28}) and (\ref{C2g42}) we have
\begin{align*}
&\dot{\nu}^{n+1}-y^2F^{ii}\na_{ii}\nu^{n+1}\\
&=-\F(1-y^2)-y(1-y^2)\sum f_i\lll(\frac{\kappa_i}{y}-1\rrr)-y^2\sum F^{ii}y_l\tna_l\thh_{ii}-y\sum\nu^{n+1}_ly_l\sum f_i\\
&+y^2\sum F^{ii}y_l\tna_l\thh_{ii}+y^3\sum F^{ii}\thh^2_{ii}-2y\sum F^{ii}y_i\nu^{n+1}_i+y\sum\nu^{n+1}_ly_l\sum f_i\\
&=-\F(1-y^2)-(1-y^2)\sum f_i(\kappa_i-y)+y\sum f_i(\kappa_i-y)^2-2y\sum f_iy_i\nu^{n+1}_i.
\end{align*}
\end{proof}
From the assumption
\[\frac{\partial\varphi}{\partial t}-y^2F^{ii}\na_{ii}\varphi\geq0\;\mbox{at $(\mathbf{z}_0,t_0),$}\]
we have that
\be\label{C2g30}
\begin{aligned}
0&\leq\dot{h}^1_1-y^2F^{ii}\na_{ii}h^1_1-\frac{h^1_1}{\nu^{n+1}-a}\lll(\dot{\nu}^{n+1}-y^2F^{ii}\na_{ii}\nu^{n+1}\rrr)\\
&=3\F\kappa_1^2+y^2F^{ij,kl}h^j_{i;1}h^l_{k;1}-\F+\kappa_1\sum f_i\kappa_i^2+\kappa_1\sum f_i-F-\kappa_1^2F\\
&-\frac{\kappa_1}{y-a}\left\{-\F(1-y^2)-(1-y^2)F+y(1-y^2)\sum f_i+y\sum f_i\kappa_i^2\right.\\
&\left.-2y^2F+y^3\sum f_i-2y\sum f_iy_i\nu^{n+1}_i\right\}\\
&=\left[3\kappa_1^2-1+\frac{\kappa_1(1-y^2)}{y-a}\right]\F+y^2F^{ij,kl}h^j_{i;1}h^l_{k;1}\\
&-\frac{a\kappa_1}{y-a}\lll(\sum f_i\kappa_i^2+\sum f_i\rrr)+\left[-1-\kappa_1^2+\frac{\kappa_1(1+y^2)}{y-a}\right]F\\
&+\frac{2y\kappa_1}{y-a}\sum f_iy_i\nu^{n+1}_i.
\end{aligned}
\ee
Here we used lemma \ref{C2glm0} and lemma \ref{C2glm1}.
Since
\[F^{ij,kl}h^j_{i;1}h^l_{k;1}\leq\sum_{i\neq j}\frac{f_i-f_j}{\kappa_i-\kappa_j}\lll(h^j_{i;1}\rrr)^2
\leq2\sum_{i\geq 2}\frac{f_i-f_1}{\kappa_i-\kappa_1}\lll(h^1_{i;1}\rrr)^2\]
and
\[h^1_{i;1}=\na_1\lll(g^{1k}h_{ki}\rrr)=y^2h_{1i;1}=y^2h_{11;i}=\frac{\kappa_1}{y-a}\nu^{n+1}_i,\]
we get
\be\label{C2g31}
F^{ij,kl}h^j_{i;1}h^l_{k;1}\leq\frac{2\kappa_1^2}{(y-a)^2}\sum_{i\geq2}\frac{f_i-f_1}{\kappa_i-\kappa_1}\lll(\nu^{n+1}_i\rrr)^2.
\ee

Therefore
\be\label{C2g32}
\begin{aligned}
&0\leq\left[3\kappa_1^2-1+\frac{\kappa_1(1-y^2)}{y-a}\right]\F+\frac{2y^2\kappa_1^2}{(y-a)^2}\sum_{i\geq2}\frac{f_i-f_1}{\kappa_i-\kappa_1}\lll(\nu^{n+1}_i\rrr)^2\\
&-\frac{a\kappa_1}{y-a}\lll(\sum f_i\kappa_i^2+\sum f_i\rrr)+\lll[-1-\kappa_1^2+\frac{\kappa_1(1+y^2)}{y-a}\rrr]F+\frac{2y\kappa_1}{y-a}\sum f_iy_i\nu^{n+1}_i.\\
\end{aligned}
\ee
Let
$$\begin{aligned}
&I=\{i:\kappa_i-y\leq-\theta\kappa_1\},\\
&J=\{i:-\theta\kappa_1<\kappa_i-y<0,\;f_i<\theta^{-1}f_1\},\\
&L=\{i:-\theta\kappa_1<\kappa_i-y<0,\;f_i\geq\theta^{-1}f_1\},
\end{aligned}$$
where $\theta\in(0,1)$ is to be determined later.
Then we have
\be\label{C2g33}
\begin{aligned}
\frac{-1}{y-a}\sum_{i\in I}(\kappa_i-\nu^{n+1})^2f_i&\leq\frac{\theta\kappa_1}{y-a}\sum_{i\in I}f_i(\kappa_i-\nu^{n+1})\\
&\leq\frac{\theta\kappa_1}{y-a}\sum_{i\in I}f_iy_i^2(\kappa_i-\nu^{n+1}),
\end{aligned}
\ee
provided $\theta\kappa_1a>2$ we get
\be\label{C2g34}
\begin{aligned}
-\frac{\kappa_1a}{y-a}\sum_{i\in I}f_i(\kappa_i-\nu^{n+1})^2&\leq+\frac{2\kappa_1}{y-a}\sum_{i\in I}f_iy_i^2(\kappa_i-\nu^{n+1})\\
&=-\frac{2y\kappa_1}{y-a}\sum_{i\in I}f_iy_i\nu^{n+1}_i.
\end{aligned}
\ee
\be\label{C2g35}
\sum_{i\in J}f_iy_i^2(\nu^{n+1}-\kappa_i)\geq-\theta\kappa_1\sum_{i\in J}\theta^{-1}f_1y_i^2\geq f_1\kappa_1,
\ee
provided $a\kappa_1>2,$
\be\label{C2g36}
\frac{-a\kappa_1}{y-a}f_1\kappa_1^2+\frac{2y\kappa_1}{y-a}\sum_{i\in J}f_iy_i\nu^{n+1}_i<0.
\ee
Finally, when $i\in L,$
\be\label{C2g37}
\begin{aligned}
&\frac{2y\kappa_1}{y-a}\sum_{i\in L}f_iy_i\nu^{n+1}_i-\frac{2y^2\kappa_1^2}{(y-a)^2}\sum_{i\in L}\frac{f_i-f_1}{\kappa_1-\kappa_i}\lll(\nu^{n+1}_i\rrr)^2\\
&\leq\frac{-2\kappa_1}{y-a}\left[\sum_{i\in L}f_iy_i^2(\kappa_i-\nu^{n+1})+\frac{1}{y-a}\sum_{i\in L}\frac{1-\theta}{1+\theta}f_iy_i^2(\kappa_i-\nu^{n+1})^2\right]\\
&=\frac{-2\kappa_1}{y-a}\sum_{i\in L}f_iy_i^2\lll[(\kappa_i-y)+\frac{1-\theta}{(y-a)(1+\theta)}(\kappa_i-y)^2\rrr]\\
&\leq\frac{\kappa_1(1+\theta)(1-y^2)}{2(1-\theta)}\sum_{i\in L}f_i.
\end{aligned}
\ee
We want $\frac{\kappa_1(1+\theta)(1-y^2)}{2(1-\theta)}-\frac{a\kappa_1}{y-a}\leq 0,$
which is equivalent to
\be\label{C2g38}
\phi_{\theta}(y)=a-\frac{(1+\theta)(1-y^2)(y-a)}{2(1-\theta)}\geq0\;\;\mbox{on $y\in(a,1]$}.
\ee
Since
\be\label{C2g39}
\begin{aligned}
&\phi_0(y)=a-\frac{1}{2}(1-y^2)(y-a)\\
&>\frac{4}{3}a-\frac{1}{27}a^3-\frac{1}{27}(a^2+3)^{\frac{3}{2}}:=\phi(a).
\end{aligned}
\ee
For $a\in(0,1)$ it is easy to check that $\phi'(a)>0,$ $\phi(0)<0,$ $\phi(1)>0.$ Let $\sigma_0$ be the unique zero of $\phi(a)$ in $(0,1).$
Numerical calculation show that $0.14596<\sigma_0<0.14597.$

\section{Convergence to a stationary solution} \label{Con}
Let us go back to our original equation (\ref{Int1}), which is a scalar parabolic differential equation defined on the cylinder $\Omega_T=\Omega\times[0,T)$ with initial value $u(0)=u_0\in C^{\infty}(\Omega).$ In view of the a priori estimates, which we have estimated in the preceding sections, we know that
\be\label{Con0}
\sqrt{1+|Du|^2}\leq \frac{1}{\sigma}.
\ee
and when $\sigma>\sigma_0$ $(0.14596<\sigma_0<0.14597)$ there is a constant $C$ independent of $\epsilon$ and $t$ such that
\be\label{Con5}
u|D^2u|\leq C.
\ee
Thus we have
\be\label{Con1}
\mbox{$F$ is uniformly elliptic in $u$.}
\ee
Moreover, since $F$ is concave, we have uniform $C^{2+\alpha}(\Omega)$ estimates for $u^2(t),\;\forall t>0.$ Therefore, the flow exists for all $t\geq 0.$

By integrating equation (\ref{Int1}) with respect to $t$, we get
\be\label{Con2}
u(x,t)-u(x,0)=\int_0^t \F uw
\ee
which implies
\be\label{Con3}
\int_0^{\infty}\F uw<\infty\;\;\mbox{$\forall x\in\Omega.$}
\ee
Hence, for any $x\in\Omega$ there is a sequence $t_k\rightarrow\infty$ such that $\F u(x)\rightarrow 0.$

On the other hand, due to our assumptions on our initial surface, $u(x,\cdot)$ is monotone decreasing and therefore
\be\label{Con4}
\lim_{t\rightarrow\infty}u(x,t)=\tilde{u}(x)
\ee
exists, and is of class $C^{\infty}(\Omega).$ So $\tilde{u}(x)$ is a stationary solution of our problem.

\section*{Acknowledgement}
The author would like to thank Professor Joel Spruck for his guidance and support.

\end{document}